\documentclass[12pt,reqno]{amsart}
\usepackage{fullpage,enumerate,colonequals}
\usepackage{pdfsync}
\usepackage{amsfonts, amsmath, amstext, amssymb, amsthm}
\usepackage{latexsym}

\usepackage{url}

\usepackage{mathtools}
\usepackage{enumerate}

\newcommand{\qq}{\mathfrak{q}}
\newcommand{\mm}{\mathfrak{m}}

\newcommand{\OO}{\mathcal{O}}
\newcommand{\RR}{\mathbb{R}}
\newcommand{\QQ}{\mathbb{Q}}
\newcommand{\ZZ}{\mathbb{Z}}
\newcommand{\FF}{\mathbb{F}}
\newcommand{\pp}{\mathfrak{p}}

\newcommand{\PP}{\mathbb{P}}
\newcommand{\II}{\mathbb{I}}

\DeclareMathOperator{\red}{red}
\DeclareMathOperator{\Nrd}{Nrd}
\DeclareMathOperator{\Gal}{Gal}
\DeclareMathOperator{\Cl}{Cl}
\DeclareMathOperator{\ch}{char}
\DeclareMathOperator{\coker}{coker}
\newtheorem{thm}{Theorem}[section]
\newtheorem{prop}[thm]{Proposition}
\newtheorem{lem}[thm]{Lemma}
\newtheorem{cor}[thm]{Corollary}
\theoremstyle{definition}
\newtheorem{defn}[thm]{Definition}
\newtheorem{rmk}[thm]{Remark}

\title{Universally and existentially definable subsets of global fields}

\author{Kirsten Eisentr\"{a}ger}
  \address{Department of Mathematics \\ The Pennsylvania State
  University \\ University, Park, PA 16802, USA}
  \email{eisentra@math.psu.edu}
  \urladdr{http://www.personal.psu.edu/kxe8/}
     
\author{Travis Morrison}
  \address{Department of Mathematics \\ The Pennsylvania State
  University \\ University, Park, PA 16802, USA}
  \email{txm950@psu.edu}
  \urladdr{http://www.personal.psu.edu/txm950/}

\begin{document}

\begin{abstract}
  We show that rings of $S$-integers of a global function field $K$ of odd
  characteristic are first-order universally definable in $K$. This extends
  work of Koenigsmann and Park who showed the same for $\ZZ$ in $\QQ$ and the ring of integers in a
  number field, respectively.
  We also give another proof of a theorem of Poonen and show that the
  set of non-squares in
  a global field of characteristic $\not=2$ is diophantine. Finally, we show that the set of pairs
  $(x,y)\in K^{\times}\times K^{\times}$ such that $x$ is not a norm
  in $K(\sqrt{y})$ is diophantine over $K$ for any global field $K$ of characteristic $\not=2$.
\end{abstract}

\maketitle

\section{Introduction}

Hilbert's Tenth Problem asks whether there exists an algorithm that
decides, given an arbitrary polynomial equation with integer
coefficients, whether it has a solution in the integers. Matiyasevich
answered this in the negative in \cite{Mat70} using work by Davis,
Putnam, and J.\ Robinson \cite{DPR61}. We say that Hilbert's Tenth
Problem is undecidable. The same question can be asked for polynomial
equations with coefficients and solutions in other commutative rings
$R$. We refer to this as Hilbert's Tenth Problem over $R$. Hilbert's
Tenth Problem over $\QQ$, and over number fields in general, is still
open. The function field analogue is much better understood, and
Hilbert's Tenth Problem is known to be undecidable for global function
fields (\cite{Ph91}, \cite{Vi94}, \cite{Sh96}, \cite{Eis03}).

One approach to proving that Hilbert's Tenth Problem for $\QQ$ is
undecidable is to show that $\ZZ$ is \emph{diophantine} over $\QQ$:
\begin{defn}
  Let $R$ be a ring. We say that $A\subseteq R^m$ is diophantine over
  $R$ if there exists a polynomial
  $g(x_1,\ldots,x_m,y_1,\ldots,y_n)\in
  R[x_1,\ldots,x_m,y_1,\ldots,y_n]$ such that
\[
(a_1,\ldots,a_m)\in A \iff \exists r_1,\ldots,r_n \in R \text{ s.t. } g(a_1,\ldots,a_m,r_1,\ldots,r_n)=0.
\]
\end{defn}

If one had a diophantine, i.e.\ a positive existential definition of $\mathbb{Z}$ in $\mathbb{Q}$,
then a reduction argument, together with Matiyasevich's theorem for
Hilbert's Tenth Problem over $\mathbb{Z}$, would imply that Hilbert's
Tenth Problem for $\mathbb{Q}$ is undecidable.
But it is still open whether $\mathbb{Z}$ is positive existentially definable
in $\mathbb{Q}$. In fact, if Mazur's conjecture holds, then $\ZZ$ is
not existentially definable in $\QQ$ \cite{Maz92}.

It is, however, possible to define the integers inside the rationals
with a first-order formula. This is due to J.\ Robinson \cite{Rob49}
who gave a $\forall \exists \forall$ definition of $\mathbb{Z}$ in
$\mathbb{Q}$. Her result was improved by Poonen  \cite{Poo09} who gave a
$\forall \exists$ definition of $\mathbb{Z}$ in $\mathbb{Q}$.
 Koenigsmann \cite{Koe13} further improved on Poonen's result and
gave a definition of the integers inside $\mathbb{Q}$ that uses only
universal quantifiers.

 Park generalized this and
showed that for any number field $K$, the ring of integers $\OO_K$ is
universally definable in $K$ \cite{Park}.

Similar definability questions can be asked for subrings of global
function fields.  Let $q$ be a power of a prime. While Hilbert's
Tenth Problem for both $\FF_q[t]$ and $\FF_q(t)$ is undecidable
(\cite{Den79}, \cite{Ph91}, \cite{Vi94}), it is not known whether
$\FF_{q}[t]$ is diophantine over $\FF_{q}(t)$. Showing this still
seems out of reach, but it is possible to give a universal definition
of $\FF_q[t]$ in $\FF_q(t)$ which we do in this paper for odd $q$. More generally,
we prove the natural generalization of Park's result for defining rings of integers to global function
fields $K$.

This is the first of three main theorems in this paper, and we prove
it in Section~\ref{univdef}:
\begin{thm}\label{univdefthm}
  Let $K$ be a global function field of odd characteristic and let $S$
  be a finite, nonempty set of primes of $K$. Then $\OO_S$ is first-order
  universally definable in $K$. Equivalently, $K\setminus \OO_S$ is
  diophantine over $K$.
\end{thm}

Here, for a finite set $S$ of primes of $K$ we denote by $\OO_S$ the ring
\[
\OO_S:=\{x\in K: v_{\pp}(x)\geq 0 \quad \forall \text{ primes }\pp \not\in S\}.
\]

Our theorem generalizes a result by Rumely \cite{Rum80} who gave a
first-order definition of a polynomial ring inside a global function
field $K$. It also improves one of the results in \cite{Shl14}, where it was
shown how to define the $S$-integers in a global function
field using a first-order formula that involves one change in quantifiers.

One of the main ideas in proving Theorem~\ref{univdefthm} is to use certain
diophantine rings, parame\-trized by $K^{\times}$, to encode integrality
at finite sets of primes; this is based on ideas of Poonen in
\cite{Poo09} and Koenigsmann in \cite{Koe13}. Park replaces the
congruence classes that Koenigsmann used for
$\ZZ$ in $\QQ$ with ray classes of a fixed modulus of $K$ for a fixed
biquadratic extension of $K$. 

Some parts of Park's arguments do not extend to the function field
setting and require a different approach; in \cite{Park}, a
biquadratic extension of $K$ is chosen so that it is linearly disjoint
from the Hilbert class field of $K$.  Since the Hilbert class field of
a global function field is an infinite extension, we cannot use it in
our arguments, but there is another natural finite extension of the
global function field $K$ that we can use instead. We use a finite
extension of $K$ whose Galois group is the ideal class group of the
Dedekind domain $\OO_{S'}$ for some carefully chosen set of primes
$S'$.  Another crucial ingredient in the proof of
Theorem~\ref{univdefthm} is showing that any class in the ray class
group of a Dedekind domain $A$ contained in $K$ contains infinitely
many primes of $A$.

We also show that other arithmetic subsets of a global field $K$ are
diophantine over $K$, extending the results in
\cite{Koe13}. 
Given $y\in K^{\times}\setminus K^{\times2}$, consider the
norm map
 \begin{align*}
 N_y: K(\sqrt{y})&\to K \\ 
 a+b\sqrt{y} &\mapsto a^2-yb^2.
 \end{align*}
In Section~\ref{section:non-norms} we show the following new result:
\begin{thm}\label{nonnorm}
Let $K$ be a global field with $\ch(K)\not=2$. Then
\[
\{(x,y)\in K^{\times}\times K^{\times} | x\not\in N_y(K(\sqrt{y}))\}
\]
is diophantine over  $K$. 
\end{thm}
This generalizes a result of \cite{Koe13} from $K=\QQ$ to global
fields. 

We
also give a new proof of the following theorem:
\begin{thm}\label{nonsq}
Let $K$ be a global field with $\ch(K)\not=2$. Then $K^{\times}\setminus K^{\times2}$ is diophantine over  $K$. 
\end{thm}
This was established by Poonen in \cite{Poo09b}, using results on the
Brauer-Manin obstruction. For number fields, this was extended to
non-$n$th powers in \cite{CTG15} and further in
\cite{Dit16}.  In \cite{Koe13}, Koenigsmann gave a more
elementary proof for $K=\QQ$.
 Using
results in \cite{Park} and their extensions in this paper, together
with Artin Reciprocity and the Chebotarev Density Theorem, we
give a different proof of Poonen's result.

\section{$S$-integers and class field theory of global function fields}\label{ffcft}

Let $K$ be a global function field. In this section we recall some
facts about rings of $S$-integers and their class groups. 

\subsection{Background and definitions}
Let $K$ be a global function field, and let $S_K$ denote the set of
all primes of $K$. By a prime of $K$ we mean an equivalence class of
nontrivial absolute values of $K$. In a global function field, all
such absolute values are non-archimedean and we can represent a prime
as a pair $(\pp,\OO_{\pp})$ where $\OO_{\pp}$ is a local ring of $K$
and $\pp\subseteq\OO_{\pp}$ is its maximal ideal. We will often refer
only to the ideal $\pp$ as a prime of $K$.  We denote by $v_{\pp}$ 
the associated normalized discrete valuation on $K$. We now recall some
facts about the ring of $S$-integers $\OO_S$ in $K$ where $S\subset
S_K$ is a finite set of primes of $K$. The ring $\OO_S$ is a Dedekind
domain and its prime ideals are in one-to-one correspondence with the
primes of $K$ not in $S$; this correspondence is given by
\[
\pp\mapsto \pp\cap\OO_S.
\]
See \cite[Theorem 14.5]{Rosen02}. So given $q\in \OO_S$, we can factor
$q\OO_S$ uniquely into a product of prime ideals of $\OO_S$. The
support of the divisor of $q$ contains the primes in this
factorization and some primes of $S$.

A modulus $\mm$ of $K$ is a formal product of primes of $K$,
$\mm= \prod \pp^{\mm(\pp)}$, such that $\mm(\pp)\geq0$ for all $\pp$
and $\mm(\pp)=0$ for all but finitely many $\pp$. The group of
fractional ideals $I$ of $K$ is the free abelian group generated by
the primes of $K$. We define $I_{\mm}$ to be the subgroup generated by
the primes which do not appear in $\mm$. Given $\alpha\in K^{\times}$,
it defines a fractional ideal
\[
(\alpha):=\prod_{\pp}\pp^{v_{\pp}(\alpha)}.
\]

Define
\[
K_{\mm,1}=\{x\in K^{\times}: v_{\pp}(x-1)\geq \mm(\pp)
\text{ for all }\pp \text{ dividing } \mm\}.
\]
We have a well-defined map $i: K_{\mm,1}\to I_{\mm}$, sending $\alpha$
to  the fractional ideal $(\alpha)$. Define $P_{\mm}$ to be the subgroup $i(K_{\mm,1})$ of
$I_{\mm}$. The ray class
group modulo $\mm$, $C_{\mm}$, is defined to be the quotient $I_{\mm}/P_{\mm}$. This
group is not finite, but the degree map (which we can define by
viewing these formal products as divisors) induces an exact sequence
\[
0\to C^0_{\mm}\to C_{\mm}\to \ZZ\to 0.
\]
 The subgroup $C^0_{\mm}$ of degree-zero divisor classes can be shown to be finite. 

Let $L$ be a finite abelian extension of $K$. Suppose $\pp\in S_K$ is 
unramified in $L$. Then we define $(\pp,L/K)\in \Gal(L/K)$ to be its
associated Frobenius automorphism. If $\mm$ is a modulus of $K$
divisible by those primes which ramify in $L$, we have the global
Artin map
\begin{align*}
\psi_{L/K}: I_{\mm} &\to \Gal(L/K) \\
\prod \pp_i^{e_i} &\mapsto \prod (\pp_i,L/K)^{e_i}.
\end{align*}

\begin{thm}\label{Arep}(Artin Reciprocity)
  The Artin map is surjective and there exists a modulus $\mm$
  containing all the primes of $K$ which ramify in $L$ such that the kernel is
  $P_{\mm}N_{L/K}(I_L(\mm'))$; here $N_{L/K}$ is the norm map on
  fractional ideals and $\mm'$ is the modulus of $L$
   consisting of primes of $L$ lying above those of $K$ contained in $\mm$.
\end{thm}

\begin{rmk}We call any $\mm$ as in Theorem \ref{Arep} an
  {\em admissible} modulus for the extension $L$ over $K$.
\end{rmk}

We will need the existence theorem of class field theory for function
fields, so we introduce the idele group. We define the idele group of
$K$ to be the restricted product of $K_{\pp}^{\times}$ for
$\pp\in S_K$ with respect to the compact groups $R_{\pp}^{\times}$,
where $R_{\pp}$ is the ring of integers in the completion $K_{\pp}$ of
$K$ at $\pp$. We denote the idele group of $K$ by $\II$.  Denote the
diagonal embedding of $K^{\times}$ in $\II$ again by $K^{\times}$ and
the idele class group by $C_K:=\II/K^{\times}$, which we endow with
the quotient topology.

An idele of $K$ determines a fractional ideal of $K$ via the surjective map
\begin{align*}
\text{id}: \II &\to I \\
(x_{\pp})&\mapsto \prod_{\pp\in S_K} \pp^{v_{\pp}(x_{\pp})}.
\end{align*}

Again let $L/K$ be a finite abelian extension and let $S_{ram}$ be the set of primes of $K$ ramifying in $L$. Define
\[
\II_{S_{ram}}:=\{(x_{\pp})\in \II: x_{\pp}=1 \text{ for all } \pp \in S_{ram}\}.
\]

There is a unique function $\phi_{L/K}: \II \to \Gal(L/K)$ which is
continuous, trivial on $K^{\times}$, and satisfies
$\phi_{L/K}((x_{\pp}))=\psi_{L/K}(\text{id}((x_{\pp})))$ for all
$(x_{\pp})\in \II_{S_{ram}}$; see  \cite{Ta67}. Hence $\phi_{L/K}$ induces a map on the idele class
group $C_K$, which we again denote by $\phi_{L/K}$. This is the idelic
Artin map. The idelic Artin reciprocity law states that the kernel of
$\phi_{L/K}$ is $N(C_L)$ where $N$ is the norm map on ideles.

Now we state the existence theorem of class field theory; see Theorem 1 of Chapter 8 in \cite{ATcft}. 

\begin{thm}[Existence Theorem]\label{existthm}	
  Let $K$ be any global field. Fix an algebraic closure $\overline{K}$
  of $K$. Given a finite index open subgroup $H\subset C_K$, there is
  a unique finite abelian extension $L$ of $K$ in $\overline{K}$ such
  that $H=N(C_L)$.
\end{thm}

\subsection{Ray class groups of rings of $S$-integers}

Let $S_{\infty}:=\{\infty_1,\ldots,\infty_n\} \subseteq S_K$, and
$A:=\OO_{S_{\infty}}$. The ideal class group $\Cl(A)$ of the Dedekind
domain $A$ is finite. Denote by $K^A$ the maximal abelian unramified
extension of $K$ in which each prime of $S_{\infty}$ splits
completely. Then $\Cl(A)\cong \Gal(K^A/K)$ via the Artin map (see
\cite{Rosen87}), and we will use $K^A$ and $\Cl(A)$ in Section \ref{integrality}. Set $\infty:=\infty_1\cdots\infty_n$. For the rest of this section, let $\mm$ be another modulus of $K$ coprime to $\infty$; we can then view $\mm$ as an ideal of $A$. Below we define $\Cl_{\mm}(A)$, the ray class group of $A$ for the modulus $\mm$. One crucial ingredient to the universal definition of $S$-integers is Lemma \ref{Aprime}, where we need that a given class of $\Cl_{\mm}(A)$ contains infinitely many primes of $A$. This is Theorem ~\ref{chebA} at the end of this section, whose proof uses idelic class field theory. The remainder of this section is devoted to proving Theorem \ref{chebA}. One can think of this as a function field analogue of Dirichlet's theorem.

\begin{defn}
  Let $I_{\mm}(A)$ be the group of fractional ideals of $A$ which are
  coprime to $\mm$. We have a natural injection
  $j: K_{\mm,1}\to I_{\mm}(A)$ whose image we denote by
  $P_{\mm}(A)$. This image consists of principal fractional ideals
  $\alpha A$ with $v_{\pp}(\alpha-1)\geq \mm(\pp)$.  Define
  \[
  \Cl_{\mm}(A):=I_{\mm}(A)/P_{\mm}(A).
  \]
\end{defn}

We need the following lemma:

\begin{lem}[Kernel-Cokernel sequence]
Given a pair of maps between abelian groups 
\[
A \xrightarrow{f} B \xrightarrow{g} C
\]
there is an exact sequence 
\[
0 \to \ker f \to \ker g\circ f \to \ker g \to \coker f \to \coker g\circ f \to \coker g \to 0.
\]
\end{lem}
\begin{proof}
This is
Proposition 0.24 in \cite{milne2006}.
\end{proof}

Define 

\begin{align*}
U_{\mm,\pp}&:=\{x\in K_{\pp}^{\times}: v_{\pp}(x-1)\geq \mm(\pp)\},\\ 
\II_{\mm}&:= \prod_{\pp\nmid \mm} K_{\pp}^{\times} \times \prod_{\pp|\mm} U_{\mm,\pp} \cap \II,\\ 
W_{\mm}&:= \prod_{\pp|\infty} K_{\pp}^{\times} \times \prod_{\pp\nmid m\cdot\infty} \OO_{\pp}^{\times} \times \prod_{\pp|\mm} U_{\mm,\pp}.
\end{align*}

\begin{prop}\label{raysasideles}
Let $\mm$ be a modulus of $K$ coprime to $\infty$. 
\begin{enumerate}
\item The ideal map $\text{id}: \II_\mm \to I_{\mm}(A)$ induces an isomorphism $\II_{\mm}/(K_{\mm,1}W_{\mm})\cong \Cl_{\mm}(A)$.
\item The inclusion map $\II_{\mm}\to \II$ induces an isomorphism $\II_{\mm}/K_{\mm,1}\cong C_K$.
\end{enumerate}
\end{prop}

\begin{proof}
(1) We have maps
\[
K_{\mm,1}\to \II_{\mm}\to I_{\mm}(A),
\]
where the first map is the diagonal embedding and the second is the ideal map. Applying the kernel-cokernel sequence gives us the exact sequence
\[
W_{\mm} \to \II_{\mm}/K_{\mm,1}\to \Cl_{\mm}(A)\to 0.
\]
This follows because by definition, $\Cl_{\mm}(A)=I_{\mm}(A)/P_{\mm}(A)$, $P_{\mm}(A)
=j(K_{\mm,1})$, and the kernel of the ideal map restricted to
$\II_{\mm}$ is $W_{\mm}$. Thus $\Cl_{\mm}(A)$ is isomorphic to
$\II_{\mm}/K_{\mm,1}$ modulo the image of $W_{\mm}$. Since the
first map in the above sequence is reduction modulo the image of
$K_{\mm,1}$ in $W_{\mm}$, we get that
\[
\Cl_{\mm}(A)\cong (\II_{\mm}/K_{\mm,1})/(W_{\mm}K_{\mm,1}/K_{\mm,1})\cong \II_{\mm}/K_{\mm,1}W_{\mm}.
\]

(2) The injection $\II_{\mm}\to \II$ gives us an injection
$\II_{\mm}/K_{\mm,1}\to C_K$. Denote the maximal ideal of $R_{\pp}$ by $\hat{\pp}$. Let $(x_{\pp})\in \II$ and choose, using
weak approximation, $b\in K^{\times}$ such that $v_{\pp}(x_{\pp}-b)>\mm(\pp) + v_{\pp}(x_{\pp})$ for each $\pp|\mm$. Because $v_{\pp}(x_{\pp}-b)>v_{\pp}(x_{\pp})$, we must have that $v_{\pp}(x_{\pp})=v_{\pp}(b)$. Then 
\[
v_{\pp}(x_{\pp}/b-1)=v_{\pp}(x_{\pp}-b)-v_{\pp}(b)>\mm(\pp),
\]
 implying that $x_{\pp}/b\in U_{\mm,\pp}$ for each
$\pp|\mm$. Then $(x_{\pp}/b)\in \II_{\mm}$ maps to the image of
$(x_{\pp})$ in $C_K$ and we see that the map is surjective.
\end{proof}

\begin{prop}\label{Clmafinite}
The group $\Cl_{\mm}(A)$ is finite.
\end{prop}
\begin{proof}
 In any
   ideal class of $A$, we can find an ideal in it coprime to $\mm$, so we have a surjection $\Cl_{\mm}(A) \to \Cl(A)$. The kernel of this
  map is the subgroup of $\Cl_{\mm}(A)$ consisting of the principal
  ideal classes of the form $\overline{xA}$ where $v_{\pp}(x)=0$ for $\pp|\mm$. These classes can be viewed as elements in
  $C_{\mm}^0$. Since $C^0_{\mm}$ is finite, we see that $\Cl_{\mm}(A)$
  is an extension of finite groups and is itself finite.
\end{proof}

\begin{cor}
There is a finite abelian extension $K_{\mm}^A$ of $K$ whose Galois group is isomorphic to $\Cl_{\mm}(A)$ via the Artin map. 
\end{cor}

\begin{proof}
The injection $\II_{\mm}\to \II$, when restricted to $K_{\mm,1}W_{\mm}$, induces an isomorphism 
\[
(K_{\mm,1}W_{\mm})/K_{\mm,1}\simeq (K^{\times}W_{\mm})/K^{\times}.
\]
Indeed, given $c\in K^{\times}$ and $(x_{\pp})\in W_{\mm}$ we can use weak approximation to find $c'\in K^{\times}$ such that $c/c'\in K_{\mm,1}$ as in the proof of part (2) of Theorem \ref{raysasideles}. Then $\frac{c}{c'}(x_{\pp})\in K_{\mm,1}W_{\mm}$ maps to the image of $c\cdot(x_{\pp})$ in $(K^{\times}W_{\mm})/K^{\times}$. Thus $\Cl_{\mm}(A)$ is a quotient of $C_K$ by an open subgroup
  subgroup, whose index is finite by Proposition \ref{Clmafinite}. The Existence Theorem then guarantees such an extension $K_{\mm}^A/K$ as desired.
\end{proof}

 Together with an application of the Chebotarev Density Theorem, the above
  discussion gives us the following theorem:
  
\begin{thm}\label{chebA}
Let $A$ and $\mm$ be as above. Any ideal class in $\Cl_{\mm}(A)$ contains infinitely
  many prime ideals of $A$.
\end{thm}

\section{$S$-integers are universally definable in global function fields}\label{univdef}

In Sections \ref{QAfacts} and \ref{integrality} we will review facts and results from \cite{Park}. The rest of Park's argument does not extend to the function field setting, so we use results of Section \ref{ffcft} on idelic class field theory to finish the proof in Section \ref{integralatsigma} and Section \ref{univdefthmpf}.

\subsection{Notation and facts about quaternion algebras}\label{QAfacts}

 Throughout this section, let $K$ be a global function field
of odd characteristic. Let $a,b\in K^{\times}$. We recall the following notation from \cite{Park}:

\begin{enumerate}
\item Given a prime $\pp\in S_K$, let $v:=v_{\pp}$ be its associated
  valuation, normalized so that it takes values in
  $\ZZ\cup\{\infty\}$. Define $K_{\pp}$ to be the completion of $K$ at
  $\pp$, $R_{\pp}$ the ring of integers in $K_{\pp}$, the maximal
  ideal of $R_{\pp}$ by $\hat{\pp}$, and $\FF_{\pp}$ the residue field
  of $\pp$. Set
  $U_{\pp}:=\{s\in \FF_{\pp}: x^2-sx+1 \text{ is irreducible over }
  \FF_{\pp}\}$ and let $\red_{\pp}:R_{\pp}\to \FF_{\pp}$ be the
  reduction map. Let $\OO_{\pp}:=R_{\pp}\cap K$; this is the local
  ring of the prime $\pp$ in $K$.
\item $H_{a,b}=K\oplus \alpha K \oplus \beta K \oplus \alpha\beta K$, the quaternion algebra over $K$ with multiplication given by $\alpha^2=a,\beta^2=b,\alpha\beta=-\beta\alpha$. 
\item Given $x:=x_1+x_2\alpha+x_3\beta +x_4\alpha\beta$, define $\overline{x}:=x_1-x_2\alpha-x_3\beta-x_4\alpha\beta$. This is the standard involution on $H_{a,b}$. Define the (reduced) trace of $x$ to be $x+\overline{x}=2x_1$ and the (reduced) norm of $x$ to be $x\cdot \overline{x}=x_1^2-ax_2^2-bx_3^2+abx_4^2$. 
\item $\Delta_{a,b}=\{\pp\in S_K: H_{a,b}\otimes_K K_{\pp}\not\cong M_2(K_{\pp})\}$, that is, the set of primes where $H_{a,b}$ ramifies. 
\item $(a,b)_{\pp}=\begin{cases} 1 &: \pp \not\in \Delta_{a,b} \\ -1  &: \pp\in \Delta_{a,b}\end{cases}$, the Hilbert symbol of $K_{\pp}$. 
\item $S_{a,b} = \{2x_1\in K: \exists x_2,x_3,x_4 \text{ such that } x_1^2-ax_2^2-bx_3^2+abx_4^2=1\}$. This is the set of traces of norm one elements of $H_{a,b}$. 
\item $T_{a,b} = S_{a,b}+S_{a,b}$.
\end{enumerate}

Given a prime $\pp\in S_K$, we similarly define $S_{a,b}(K_{\pp})$ and $T_{a,b}(K_{\pp})$ just by replacing $K$ with $K_{\pp}$ in the above definitions.

\begin{lem}\label{facts}
$\left.\right.$
\begin{enumerate}
\item If $\pp\not\in \Delta_{a,b}$, then $S_{a,b}(K_{\pp})=K_{\pp}$.
\item If $\pp\in \Delta_{a,b}$ then $\red_{\pp}^{-1}(U_{\pp})\subseteq S_{a,b}(K_{\pp})\subseteq \OO_{\pp}$. 
\item For any $\pp$ with $|\FF_{\pp}|>11$, we have $\FF_{\pp}=U_{\pp}+U_{\pp}$.
\item For each $a,b\in K^{\times}$, 
\[
S_{a,b}=K\cap \bigcap_{\pp\in \Delta_{a,b}} S_{a,b}(K_{\pp}).
\]
\end{enumerate}
\end{lem}

\begin{proof}
See \cite{Park}, Lemma 2.2; the arguments work when $K$ is any global field. 

\end{proof}

\begin{prop}
For any $a,b\in K^{\times}$, we have
\[
T_{a,b}=\bigcap_{\pp\in \Delta_{a,b}} \OO_{\pp}.
\]
\end{prop}

\begin{proof}
See \cite[Proposition 2.3]{Park}.
\end{proof}

In our setting, we also have the following formula for the Hilbert symbol from \cite[XIV.3.8]{Ser79}:,
\begin{equation}\label{hsymb}
  (a,b)_{\pp}= \left[(-1)^{v_{\pp}(a)v_{\pp}(b)}\left(\frac{a^{v_{\pp}(b)}}{b^{v_{\pp}(a)}}\right)\right]^{(|\FF_{\pp}|-1)/2}.
\end{equation}

Hence for a $\pp$-adic unit $a$, $(a,p)_{\pp}=-1$ if and only if
$v_{\pp}(p)$ is odd and $\red_{\pp}(a)$ is not a square in
$\FF_{\pp}$.

\subsection{Partitioning the primes of $K$}\label{integrality}
Suppose $a$ is not a square in $K$. Then if $\pp$ is unramified in
$K(\sqrt{a})$, after possibly multiplying by a square of $K^{\times}$, $a$ is a $\pp$-adic unit. We can then identify
$\psi_{K(\sqrt{a})/K}(\pp)$ with the power residue symbol
$\left(\frac{a}{\pp}\right)_2$ (see Proposition 10.5 and 10.6 in
\cite{Rosen02}).  If the images of $a,b$ are distinct in $K^{\times}/K^{\times2}$,
we can use the splitting of $\pp$ in $\Gal(K(\sqrt{a},\sqrt{b})/K)$ to
study the Hilbert symbols $(a,p)_{\pp}$ and $(b,p)_{\pp}$, since if
$v_{\pp}(p)$ is odd, $(a,p)_{\pp}=\left(\frac{a}{\pp}\right)_2$. We have the following lemma:

\begin{lem}\label{Park3.8}
  Take $a,b\in K^{\times}$ whose images in $K^{\times}/K^{\times2}$
  are distinct and let $\mm$ be an admissible modulus for the
  extension $L:=K(\sqrt{a},\sqrt{b})/K$ and let
\[
\psi_{L/K}: I_{\mm}\to\{\pm1\}\times \{\pm 1\}
\]
be the Artin map. Suppose that $\mm$ is
  divisible also by all primes dividing $(ab)$. For a prime
  $\pp\in S_K$ such that $\pp\nmid \mm$,
  $\pp\in \Delta_{a,p}\cap \Delta_{b,p}$ if and only if $v_{\pp}(p)$
  is odd and $\psi_{L/K}(\pp)=(-1,-1)$.
\end{lem}
\begin{proof}
  See \cite{Park}, Lemma 3.8. The only change in the argument is that
  we do not to worry about total positivity conditions defining
  $K_{\mm,1}$ since there are no Archimedean primes of a global
  function field.
\end{proof}

As in \cite{Park}, we partition primes of $K$ coprime to $\mm$ based on their image under $\psi_{L/K}$. Choose $a,b\in K^{\times}$ and $L$ as in Lemma \ref{Park3.8} and set
\[
\PP(p):=\{\pp \in S_K: v_{\pp}(p) \text { is odd}.\}.
\]
Also, for $(i,j)\in \Gal(L/K)$, $i,j\in \{\pm1\}$, set 
\[
\PP^{(i,j)}=\{\pp\in S_K:\pp\in I_{\mm} \text{ and }\psi_{L/K}(\pp)=(i,j)\}
\]
and set 
\[
\PP^{(i,j)}(p)=\PP(p)\cap\PP^{(i,j)}. 
\]

\begin{lem}\label{uptomodulus}
  Suppose $p\in K^{\times}$ and let $\mm$ be a modulus as in Lemma
  \ref{Park3.8}. Additionally, suppose $(p)$ and $\mm$ are
  coprime. Then we have the following identification of sets of
  primes, where the two sets differ at most by primes dividing the
 modulus.
\begin{align*}
\PP^{(-1,-1)}(p) &\leftrightarrow \Delta_{a,p}\cap \Delta_{b,p}, \\
\PP^{(-1,1)}(p) &\leftrightarrow \Delta_{a,p}\cap \Delta_{ab,p},  \\ 
\PP^{(1,-1)}(p) &\leftrightarrow \Delta_{b,p}\cap \Delta_{ab,p}.\\
\end{align*}
\end{lem}
\begin{proof}
See \cite[Lemma 3.9]{Park}. 
\end{proof}




\begin{lem}
  Let $\pp\in S_K$ with $\pp\nmid \mm$ and suppose that
  $\psi_{L/K}(\pp)=(i,j)$ with $(i,j)\not=(1,1)$. Then
  $\pp\in \PP^{(i,j)}(p)$ for some $p\in K^{\times}$. Hence there
  exist $u,v\in K^{\times}$ such that
  $\pp\in \Delta_{u,p}\cap \Delta_{v,p}$. If $\psi_{L/K}(\pp)=(1,1)$
  then there exist $p,q\in K^{\times}$ so that
  $\pp\in \Delta_{ap,q}\cap \Delta_{bp,q}$.
\end{lem}
\begin{proof}
See \cite[Lemmas 3.11, 3.12]{Park}.
\end{proof}

\begin{defn}\label{Jabdef}
Let $a,b\in K^{\times}$. Let 
\[
J_{a,b} := \bigcap_{\pp\in \Delta_{a,b}\cap(\PP(a)\cup \PP(b))} \pp\OO_{\pp}.
\]
\end{defn}

\begin{prop}\label{Jabdio}
$J_{a,b}$ is diophantine over  $K$.
\end{prop}
\begin{proof}
See \cite[Lemmas 3.14, 3.15, 3.17]{Park}.
\end{proof}

\subsection{Controlling integrality with diophantine sets}\label{integralatsigma}

The remainder of the argument used to give a universal definition of the $S$-integers in $K$ requires a different approach than the one in \cite{Park}. In order to make our proof work, we need to find infinitely many $q\in K^{\times}$ with prescribed image under $\psi_{L/K}$ and which generate prime ideals in a certain Dedekind domain contained in $K$ in order to control the poles of $q$. 
\begin{lem}\label{3.19}
Let $S$ be a finite set of primes in $S_K$. We can choose $a,b\in K^{\times}$ so that the following hold: 
\begin{enumerate}
\item The images of $a$ and $b$ in $K^{\times}/K^{\times2}$ are distinct.
\item Any admissible modulus  $\mm$ for $L:=K(\sqrt{a},\sqrt{b})/K$ is divisible by the primes of $S$. 
\item Given a finite set of primes $S'\subseteq S_K$ disjoint from $S$, an ideal class $\overline{I}$ in $\Cl(\OO_{S'})$, and an element $\sigma \in \Gal(L/K)$, there exists a prime $\qq$ of $K$ such that $\qq\cap\OO_{S'}$ is in the ideal class $\overline{I}$, $\qq\in I_{\mm}$, and $\psi_{L/K}(\qq)=\sigma$. 
\item As fractional ideals, $(a)$ and $(b)$ are coprime, meaning their supports are disjoint.

\end{enumerate}
\end{lem}

\begin{proof}
Set $A=\OO_{S'}$. Let $\pp_1\in S_K\setminus S$ and choose $a\in K^{\times}$ with $v_{\pp_1}(a)=0$ and $v_{\pp}(a)=1$ for $\pp\in S$; this is possible by weak approximation. 

Now we choose $b\in K^{\times}$ with $v_{\pp_1}(b)=1$ and $v_{\pp}(b)=0$ for any $\pp$ in the support of $(a)$. Then $(a)$ and $(b)$ have disjoint support. They have distinct images in $K^{\times}/K^{\times2}$ as the primes at which $a$ has odd valuation differ from the primes where $b$ has odd valuation. Thus we have established (1) and (4). Since the primes in $S$ ramify in $K(\sqrt{a})$, they also ramify in $L$ and hence will be contained in any admissible modulus for $L/K$. 

We are left with showing that (3) holds. Recall that $K^A/K$ is the maximal abelian unramified extension of $K$ in which every prime of $S'$ splits completely and for which $\Gal(K^A/K)\simeq \Cl(A)$. We claim that $K^A$ and $L$ are linearly disjoint; this will follow if we can show that none of $\sqrt{a},$ $\sqrt{b}$, and $\sqrt{ab}$ are in $K^A$. Since the primes in $S$ ramify in $K(\sqrt{a})$, $\sqrt{a}\not\in K^A$. As $\pp_1$ ramifies in $K(\sqrt{b})$ and the primes of $S$ along with $\pp_1$ all ramify in $K(\sqrt{ab})$, we have that $\sqrt{b}$, and $\sqrt{ab}$ are not in $K^A$, and the claim follows. 

We conclude 
\[
\Gal(LK^A/K)\cong\Gal(L/K)\times\Gal(K^A/K)\cong \{(\pm1,\pm1)\}\times\Cl(A).
\]

We then apply the Chebotarev density theorem to find a prime $\qq$
such that \linebreak $(\qq,LK^A/K)=(\sigma,\overline{I})\in
\Gal(L/K)\times\Gal(K^A/K)$. This yields part (c): 
\[ \sigma=(\qq,LK^A/K)|_L=(\qq,L/K)=\psi_{L/K}(\qq) {\text{ and }} \overline{I}=(\qq,LK^A/K)|_{K^A}=(\qq,K^A/K).\]

\end{proof}
For the rest of this section, fix $a,b\in K^{\times}$ as in Lemma
\ref{3.19} along with an admissible modulus $\mm$ of $K$ for
$L:=K(\sqrt{a},\sqrt{b})$.

For the rest of this section, fix $a,b\in K^{\times}$ as in Lemma
\ref{3.19} and set $L:=K(\sqrt{a},\sqrt{b})$. Next, we must fix two
additional elements of $K^{\times}$. To construct them, we need
Theorem~3.7 of~\cite{Park}, which we restate below for
convenience. This theorem lets us construct an element of $K^{\times}$
with prescribed Hilbert symbols against finitely many elements of
$K^{\times}$. 

\begin{thm}\cite[Theorem 3.7]{Park}\label{prescribesymbols}
Let $K$ be a global field, $\ch(K)\not=2$. Let $\Sigma$ denote the set
of primes of $K$, 
and let $\Lambda$ be a finite set of indices. Let $(a_i)_{i\in
  \Lambda}$ be a finite sequence of elements of $K^*$ and suppose that
$(\varepsilon_{i,\pp})_{i\in \Lambda, \pp\in \Sigma}$ is a family of
elements of $\{1,-1\}$. There exists $x\in K^*$ satisfying
$(a_i,x)_{\pp}=\varepsilon_{i,\pp}$ for all $i\in \Lambda$
and $\pp\in \Sigma$ if and only if the following conditions hold:
\begin{enumerate}
\item All but finitely many of the $\varepsilon_{i,\pp}$ are equal to
  1.
\item For all $i\in \Lambda$, we have $\prod_{\pp\in \Sigma}
  \varepsilon_{i,\pp}=1$. 
\item For every $\pp\in \Sigma$, there exists $x_{\pp}\in K^*$ such
  that $(a_i,x_{\pp})_{\pp}=\varepsilon_{i,\pp}$. 
\end{enumerate}
\end{thm}

\begin{lem}\label{choosec,d}
There exist $c,d\in K^{\times}$ such that $\Delta_{a,c}=\PP(a)$ or
$\Delta_{a,c}=\PP(a)\cup \{\pp_a\}$ where $\pp_a$ is coprime to $(a)$
and $(b)$, and $\Delta_{b,d} =
\PP(b)$ or $\Delta_{b,d} \cup \{\pp_b\}$ where $\pp_b$ is coprime to
$(a)$, $(b)$, and $\pp_a$. 
\end{lem}
\begin{proof}
  Assume $\PP(a)$ has even cardinality. Then by
  Theorem~\ref{prescribesymbols}, there exists $c$ in $K$ such that
  $(a,c)_{\pp}=-1$ for each $\pp\in \PP(a)$, and $(a,c)_{\pp}=1$ if
  $\pp$ is not in $\PP(a)$. Indeed, there
  are an even number primes in $\PP(a)$, and for a local element we can
  take any $c_{\pp} \in K^{\times}$ such that
  $\left(\frac{c_{\pp}}{\pp}\right)=-1$. Then $(a,c_{\pp})_{\pp}=-1$
  since $v_{\pp}(a)$ is odd. The proof in the case that $\PP(a)$ has
  odd cardinality is the same by choosing $\pp_a$ coprime to $(a)$ and
  $(b)$ such that $\left(\frac{a}{\pp_a}\right)=-1$ and considering
  the set $\PP(a)\cup \pp_a$. The local element for $\pp_a$ can be
  taken to be any element $c_{\pp_a}\in \pp_a\setminus\pp_a^2$. The
  proof for the existence of $d$ is a similar argument.
\end{proof}

We now also fix $c,d\in K^{\times}$ as in Lemma~\ref{choosec,d}, along
with a modulus $\mm$ of $K$ for $L$ such that $\mm$ contains all the
primes dividing $(a),(b),(c)$, and $(d)$ and any other primes $\pp$
such that $(a,c)_{\pp}=-1$ or $(b,d)_{\pp}=-1$. 

\begin{cor}\label{cor3.20}
Let $p\in K^{\times}$ such that $(p)$ and $\mm$ are coprime. We have
\begin{align*}
\PP^{(-1,-1)}(p)&=\Delta_{a,p}\cap\Delta_{b,p}, \\
\PP^{(-1,1)}(p)&=\Delta_{a,p}\cap \Delta_{ab,p}\cap \Delta_{a,cp}, \\
\PP^{(1,-1)}(p)&=\Delta_{b,p}\cap\Delta_{ab,p}\cap \Delta_{b,dp}.
\end{align*}
\end{cor}
\begin{proof}
For the first equality, see \cite[Corollary 3.20]{Park}. We will prove
the second equality. We have 
\[
\Delta_{a,p} \cap \Delta_{ab,p} = \PP^{(-1,1)}(p)\cup \left\{\pp\in
\PP(a): \left(\frac{a}{\pp}\right) = -1\right\}, 
\]
by Lemma~\ref{uptomodulus}. Also, $\PP^{(-1,1)}(p)\subseteq
\Delta_{a,cp}$, since $(a,c)_{\pp}=1$ for any $\pp$ coprime to
$\mm$. Thus we need to compute the intersection $\Delta_{a,cp}\cap
\{\pp|\mm\}$. 
Suppose $\pp\in \PP(a)$. Then $(a,c)_{\pp}=-1$ by
our choice of $c$, so
$(a,cp)_{\pp}=-1$ if and only if $\left(\frac{p}{\pp}\right)=1$. If
$\pp|\mm$ but $\pp\not\in \PP(a)$, then $(a,p)_{\pp}=1$ by Equation~\ref{hsymb}, so
$(a,cp)_{\pp}=-1$ if and only if $(a,c)_{\pp}=-1$. In any case, we
  conclude that 
\[
\Delta_{a,p}\cap\Delta_{a,cp} \cap \{\pp|\mm\} = \emptyset,
\]
because whenever $(a,p)_{\pp}=-1$ for $\pp|\mm$, we have
$(a,c)_{\pp}=-1$ as well. Thus the second equality holds. The proof of the third equality goes
the same way by calculating 
\[
\Delta_{b,dp}\cap\{\pp|\mm\}. 
\]
\end{proof}
\begin{defn}\label{semilocalringsdef}
For $p,q\in K^{\times}$, let
\begin{align*}
R_p^{(-1,-1)} &= \bigcap_{\pp\in \Delta_{a,p}\cap \Delta_{b,p}} \OO_{\pp},\\ 
R_p^{(1,-1)} &= \bigcap_{\pp\in \Delta_{ab,p}\cap \Delta_{b,p}\cap\Delta_{a,cp}} \OO_{\pp}, \\
R_p^{(-1,1)}&= \bigcap_{\pp\in \Delta_{a,p}\cap \Delta_{ab,p}\cap\Delta_{b,dp}} \OO_{\pp}, \\
R_{p,q}^{(1,1)} &= \bigcap_{\pp\in \Delta_{ap,q}\cap \Delta_{bp,q}} \OO_{\pp}.
\end{align*}
\end{defn}


Given a finite set of primes $S\subset S_K$, in Section \ref{univdef}
we will express the $S$-integers $\OO_S$ in terms of the rings
$R_p^{\sigma}$ for $\sigma=(-1,-1),(-1,1),$ and $(1,-1)$ and
$R_{p,q}^{(1,1)}$ defined above.

\begin{defn} For each $\sigma \in \Gal(L/K)$, let 
\[
\Phi_{\sigma}=\{p\in K^{\times}: (p)\in I_{\mm}, \psi_{L/K}((p))=\sigma,\text{ and }\PP(p)\subseteq \PP^{(1,1)}\cup \PP^{\sigma}\}.
\]
\end{defn}
\begin{lem}\label{not1,1}
$\left.\right.$
\begin{enumerate}
\item For each $\sigma\in \Gal(L/K)$, $\Phi_{\sigma}$ is diophantine over  $K$. 
\item For any $p\in \Phi_{\sigma}$ and $\sigma\in
  \Gal(L/K)$ with $\sigma\not=(1,1)$, we have that
  $\PP^{(\sigma)}(p)$ is nonempty. Furthermore, the Jacobson radical
  of $R_p^{\sigma}$, denoted $J(R_p^{\sigma})$, is diophantine over
  $K$.
\item Let $\sigma\in \Gal(L/K)$ with
  $\sigma\not=(1,1)$, and let $\pp_0\nmid \mm$ be a prime of $K$
  satisfying $\psi_{L/K}(\pp_0)=\sigma$. Then there is an element
  $p\in \Phi_{\sigma}$ such that $\pp_0\in \PP^{\sigma}(p)$. In fact,
  $p$ can be chosen so that $\PP^{\sigma}(p)=\{\pp_0\}$.
\end{enumerate}
\end{lem}

\begin{proof}
 To prove (1), we first establish that $K_{\mm,1}$ is
  diophantine over $K$: this follows from the fact that $K_{\mm,1}$ is defined
  by the finitely many local conditions $v_{\pp}(a-1)\geq
  \mm(\pp)$ for $\pp|\mm$ and that the local rings $\OO_{\pp}$ are
  diophantine over $K$ for any prime $\pp$ by \cite[Lemma 3.22]{Shl94}. 
  Because a class of principal ideals in $C_{\mm}$ lies in $C^0_{\mm}$ which is
  finite, we see there are only finitely many inequivalent classes of
  the form $(p)P_{\mm}$ for $p\in K^{\times}$. Thus $\{p\in
  K^{\times}: (p)\in I_{\mm}\}$ is diophantine over $K$ as it is a finite union
  of translates of $K_{\mm}$.  Observe that $\PP(p)\subseteq
  \PP^{(1,1)}\cup \PP^{\sigma}$ is equivalent to one of the following:
  \begin{itemize}
  \item $\PP^{(1,-1)}(p)=\PP^{(-1,1)}(p)=\emptyset$, if $\sigma=(-1,-1)$;
  \item $\PP^{(-1,-1)}(p)=\PP^{(-1,1)}(p)=\emptyset$, if $\sigma=(1,-1)$;
  \item $\PP^{(-1,-1)}(p)=\PP^{(1,-1)}(p)=\emptyset$, if $\sigma=(-1,1)$;
  \item $\PP^{(1,-1)}(p)=\PP^{(-1,1)}(p)=\PP^{(-1,-1)}(p)=\emptyset$, if $\sigma=(1,1)$.
  \end{itemize}
  
  For $\tau\not=(1,1)$, we
  have that $\PP^{\tau}(p)=\emptyset$ if and only if $p\in
  K^{\times2}\cdot (R_p^{\tau})^{\times}$, and this is a diophantine
  subset of $K$. Thus for any $\sigma$, $\Phi_{\sigma}$ is the
  intersection of finitely many diophantine sets and is thus
  diophantine over $K$.


Now we prove (2). Suppose $\sigma\not=(1,1)$ and that $p\in \Phi_{\sigma}$. Then because
$\psi_{L/K}((p))=\sigma$ and $\PP(p)\subseteq \PP^{(1,1)}\cup
\PP^{\sigma}$, there must be some prime $\pp \in \PP^{\sigma}(p)$,
because otherwise we would have $\psi((p))=(1,1)$. If
$\sigma=(-1,-1)$, then we observe that
$J(R_p^{(-1,-1)})=J_{a,p}+J_{b,p}$ is diophantine over $K$ by Definition~\ref{Jabdef}
Lemma~\ref{Jabdio}. Similarly, 
$J(R_p^{(-1,1)}) = J_{a,p}+J_{ab,p}+J_{a,cp}$ and
$J(R_p^{(1,-1)})=J_{b,p}+J_{ab,p}+J_{b,dp}$ are diophantine over $K$.

We move on to (3). Suppose that
$\sigma\in \Gal(K(\sqrt{a},\sqrt{b})/K)$ with $\sigma\not=(1,1)$ and
$\pp_0\nmid \mm$ is a prime of $K$ satisfying
$\psi_{L/K}(\pp_0)=\sigma$. Let $\qq'\nmid \mm$ be a prime of $K$ with
$\psi_{L/K}(\qq')=(1,1)$ and let $S'=\{\qq'\}$. By Lemma \ref{3.19} we can
choose a prime $\qq$ of $K$ such that it represents the class of
$(\pp_0\cap \OO_{S'})^{-1}$ in $\Cl(\OO_{S'})$ and $\psi_{L/K}(\qq)=(1,1)$. Then there is
an element $p\in \OO_{S'}$ such that $(\pp_0\cap\OO_{S'})(\qq\cap\OO_{S'})=p\OO_{S'}$.

We claim that $p\in \Phi_{\sigma}$. Since $v_{\pp}(p)=0$ if $\pp\not=\pp_0,\qq,\qq'$, it follows that $(p)\in I_{\mm}$ and 
\[
\psi_{L/K}((p))=\psi_{L/K}(\pp_0)\psi_{L/K}(\qq)\psi_{L/K}(\qq')=\sigma.
\]
Additionally, the only primes $\pp$ with $v_{\pp}(p)$ odd are $\pp_0$, $\qq$, and possibly $\qq'$. Hence $\PP(p)\subseteq \PP^{(1,1)}\cup\PP^{\sigma}$ and $p\in \Phi_{\sigma}$. Finally, we observe that $\PP^{\sigma}(p)=\{\pp_0\}$. 

\end{proof}

\begin{lem}\label{Aprime}
Let $\pp_0,\qq_0$ be primes of $K$ not dividing $\mm$ with $\psi_{L/K}(\qq_0)=(1,1)$. Let $A:=\OO_{\{\qq_0\}}$. Then there exists infinitely many $q\in K^{\times}$ satisfying
\begin{enumerate}
\item $\psi_{L/K}((q))=(-1,-1)$;
\item $\left(\frac{q}{\pp_0}\right) = -1$;
\item $qA$ is a prime ideal of $A$, so there exists a prime $\qq$ of $K$ such that $\qq\cap A = qA$. 
\end{enumerate}
\end{lem}

\begin{proof}
  Set
  \[
  K_{\mm}:=\{ \alpha\in K^{\times}: v_{\pp}(\alpha)=0\quad \forall \pp
  | \mm\}.
  \]
  For a prime $\pp\not=\qq_0$ of $K$, let $I_{\pp}:=A\cap\pp$ be the
  associated prime ideal in $A$. Given $x\in K_{\mm}$, we can find $y,z\in A\cap K_{\mm}$ such that $x=y/z$. Then $y$ and $z$ are $\pp$-adic units for each $\pp|\mm$ and thus we can map $x=y/z$ to the image of $y/z$ modulo $\prod_{\pp|\mm}I_{\pp}^{\mm(\pp)}$. The kernel of this map is $K_{\mm,1}$ so there is a well-defined isomorphism
\[
K_{\mm}/ K_{\mm,1} \simeq \left(A\left/\prod_{\pp|\mm}I_{\pp}^{\mm(\pp)}\right.\right)^{\times}.
\]
By the Chinese Remainder Theorem,
\[
K_{\pp_0\mm}/K_{\pp_0\mm,1}\cong K_{\mm}/K_{\mm,1}\times (A/I_{\pp_0})^{\times}. 
\]
The group $K_{\mm}/K_{\mm,1}$ surjects onto the ray classes in $C_{\mm}$ consisting of principal fractional ideals by the map 
\[
xK_{\mm,1} \mapsto (x)P_{\mm}.
\]
Now we apply Lemma \ref{3.19} part (3) to find a prime $\qq'$ of
$K$ in the principal ideal class of $\Cl(A)$ such that
$\psi_{L/K}(\qq')=(-1,-1)$. Then there exists $x_1\in K^{\times}$ such that
$\psi_{L/K}((x_1))=(-1,-1)$ and $x_1A=\qq'\cap A$ is a prime ideal of $A$. Let $s$ be a non-square of
$(A/I_{\pp_0})^{\times}$. By the Chinese Remainder Theorem, the element $(x_1K_{\mm,1},s)$ in $K_{\mm}/K_{\mm,1}\times (A/I_{\pp_0})^{\times}$ corresponds to an element
$x_2K_{\pp_0\mm,1}\in K_{\pp_0\mm}/K_{\pp_0\mm,1}$. By construction, $\left(\frac{x_2}{\pp_0}\right)=-1$. The ideal $x_2A$ need not be a prime ideal of $A$, but  by Theorem \ref{chebA}, we can find infinitely many
prime ideals of $A$ in the class generated by $x_2$ in $\Cl_{\mm\pp_0}(A)$; such a prime ideal is of the form $qA=\qq\cap A$. Both $qA$ and $x_2A$ generate the same class in $\Cl_{\mm\pp_0}(A)$ and hence $q=x_2t$ for some $t\in K_{\mm\pp_0,1}$. Thus $\psi_{L/K}((q))=\psi_{L/K}((x_2))$ and since $q$ and $x_2$, as elements of $A$, are congruent modulo the ideal $I_{\pp_0}$ of $A$,
\[
\left(\frac{q}{\pp_0}\right)=\left(\frac{x_2}{\pp_0}\right)=-1.
\]
 Thus the element $q$ satisfies the three requirements of the lemma.

\end{proof}

We now need more definitions from \cite{Park}.
\begin{defn}
For $\sigma \in \Gal(K(\sqrt{a},\sqrt{b})/K)$, and $S\subseteq S_K$ the fixed set of primes from above, define 
\begin{align*}
\widetilde{\Phi_{\sigma}} &:= K^{\times2}\cdot \Phi_{\sigma}; \\
\Psi_K &:= \Bigg\{ (p,q)\in \widetilde{\Phi_{(1,1)}}\times \widetilde{\Phi_{(-1,-1)}} | \prod_{\pp|\mm} (ap,q)_{\pp}=-1 \\
&\quad\quad\text{ and } p\in a\cdot K^{\times2}\cdot(1+J(R_q^{(-1,-1)}))\Bigg\}. \\
\end{align*}
\end{defn}


\begin{lem}\label{1,1}
$\left. \right.$
\begin{enumerate}
\item The set $\Psi_K$ is diophantine over  $K$. 
\item For $(p,q)\in \psi_{L/K}$, we have $\emptyset\not=\Delta_{ap,q}\cap \Delta_{bp,q}\subseteq I_{\mm}$, and $J(R_{p,q}^{(1,1)})$ is diophantine over  $K$. 
\item For each prime ideal $\pp_0$ satisfying $\pp_0\nmid \mm$ and $\psi_{L/K}(\pp_0)=(1,1)$, there exists $(p,q)\in \Psi_K$ such that $\Delta_{ap,q}\cap \Delta_{bp,q}=\{\pp_0\}$. 
\end{enumerate}
\end{lem}

\begin{proof}
  By Lemma \ref{not1,1} part (1),
  $\widetilde{\Phi_{(1,1)}}\times \widetilde{\Phi_{(-1,-1)}}$ is
  diophantine over $K$, and by Lemma \ref{not1,1} part (2),
  $J(R_q^{(-1,-1)})$ is diophantine over $K$. By Theorem \ref{nonlocnorm},
  for any prime $\pp$ of $K$,
\[
\{(x,y):(x,y)_{\pp}=-1\}\subseteq K^{\times}\times K^{\times}
\] 
is diophantine over $K$. Here, $(\cdot,\cdot)_{\pp}$ denotes the Hilbert symbol of $K$ at $\pp$. Since only finitely many $\pp$ divide $\mm$, the set $\Psi_K$ is the intersection of finitely many sets diophantine over $K$. 

For (2), the proof in  \cite{Park} of Lemma 3.25 part (2) shows that for $(p,q)\in \Psi_K$, if $\pp\nmid\mm$ then $\Delta_{ap,q}\cap\Delta_{bp,q}\cap I_{\mm}$ is nonempty. We only need to show that if $\pp|\mm$, $\pp\not\in \Delta_{ap,q}\cap \Delta_{bp,q}$. Suppose that $\pp\in \Delta_{ap,q}$. Then $(ap,q)_{\pp}=-1$. Since $p\in \widetilde{\Phi_{(1,1)}}$ and $q\in \widetilde{\Phi_{(-1,-1)}}$ we can assume that, possibly after multiplying by a square of $K^{\times}$, $v_{\pp}(p)=v_{\pp}(q)=0$. Then $(ap,q)_{\pp}=-1$ implies that $v_{\pp}(ap)=v_{\pp}(a)$ is odd, which implies $v_{\pp}(b)$ must be $0$ since the fractional ideals $(a)$ and $(b)$ have disjoint support. Thus $(bp,q)_{\pp}=1$ and $\pp\not\in \Delta_{ap,q}\cap\Delta_{bp,q}$. 

Now we prove (3). Let $\pp_0$ satisfy $\pp_0\nmid \mm$ and $\psi_{L/K}(\pp_0)=(1,1)$; we wish to construct a pair $(p,q)\in \Psi_K$ such that $\Delta_{ap,q}\cap\Delta_{bp,q}=\{\pp_0\}$. We begin by constructing our candidate for $q$. Choose a different prime $\qq_0$ of $K$ not dividing $\mm$ with $\psi_{L/K}(\qq_0)=(1,1)$. Using Lemma \ref{Aprime}, choose $q\in K^{\times}$ such that $\psi_{L/K}((q))=(-1,-1)$, $\left(\frac{q}{\pp_0}\right)=-1$, and $q$ generates a prime ideal of $A:=\OO_{\{\qq_0\}}$. Thus there exists a prime $\qq$ of $K$ with $qA=\qq\cap A$. Observe that this implies that the support of $(q)$ is $\{\qq,\qq_0\}$. We claim that 
\[
\Delta_{a,q}\cap\Delta_{b,q}=\{\qq\}.
\] 
We begin by showing that $(a,q)_{\qq}=(b,q)_{\qq}=-1$. The support of $(a)$ and $(b)$ is contained in $\mm$, so $v_{\qq}(a)=v_{\qq}(b)=0$. We have that $v_{\qq}(q)$ is odd, so we need to show $a$ and $b$ are non-squares in the completion of $K$ at $\qq$. This follows immediately from $\psi_{L/K}(\qq)=(-1,-1)$.  From $\psi_{L/K}(\qq_0)=(1,1)$ we get that $\qq_0\not\in \Delta_{a,q}\cap\Delta_{b,q}$. No other prime $\pp$ in $I_{\mm}$ can appear in this set since we have $v_{\pp}(q)=v_{\pp}(a)=v_{\pp}(b)=0$. Finally, no prime  $\pp|\mm$ can occur in $\Delta_{a,q}\cap\Delta_{b,q}$ since we cannot have that $v_{\pp}(a)$ and $v_{\pp}(b)$ are both odd as their supports are disjoint, and $v_{\pp}(q)=0$. This proves the claim.

For each prime $\pp|\mm$, let $E_{\pp}\subseteq K$ be a generating set for $K_{\pp}^{\times}/K_{\pp}^{\times2}$ chosen so that for each $e\in E_{\pp}$, we have $v_{\pp_0}(e-1)\geq 0$. Fix $e_0\in K$ such that $\left(\frac{e_0}{\qq}\right)=-1$ and $\left(\frac{e_0}{\pp_0}\right)=1$.

Now one can construct $p\in K^{\times}$ so that $(p,q)\in \psi_{L/K}$ with $\Delta_{ap,q}\cap\Delta_{bp,q}=\{\pp_0\}$ as it is constructed in \cite{Park}, Lemma 3.25(c). The remainder of the proof exactly follows Park's after Equation 3.6 (loc.\ cit.). The only difference is that instances of the \emph{ideal} $(q)$ of the number field $K$ are replaced with the prime $\qq$ of $K$ in this proof.



\end{proof}

\subsection{Proof of main theorem}\label{univdefthmpf}
Our general strategy in proving Theorem \ref{univdefthm} follows that of \cite{Koe13}. 

\begin{defn} Given some finite set of primes $\Delta\subseteq S$ of $K$, consider the semi-local subring $R=\bigcap_{\pp\in \Delta} \OO_{\pp}$ of $K$. Set 
\[
\widetilde{R}=\{x\in K: \not\exists y \in J(R) \text{ with } xy=1\}.
\]
\end{defn}

\begin{lem} $\left.\right.$
\begin{enumerate}
\item If $J(R)$ is diophantine over  $K$, then $\widetilde{R}$ is defined by a universal formula in $K$. 
\item $\widetilde{R}=\bigcup_{\pp\in \Delta} \OO_{\pp}$, provided that $\Delta\not=\emptyset$. In particular, $\widetilde{\OO_{\pp}}=\OO_{\pp}$.
\end{enumerate}
\end{lem}

\begin{proof}
To see (1), observe that 
\[
\widetilde{R}=\{x\in K: x=0 \text{ or } x^{-1}\in K\setminus J(R)\}.
\]
Since $J(R)$ is diophantine, its complement is defined by a universal formula, and thus so is $\widetilde{R}$. 

Now we prove (2). To see that $\widetilde{R}\subseteq\bigcup_{\pp\in \Delta} \OO_{\pp}$, assume $x\not\in \bigcup_{\pp\in \Delta} \OO_{\pp}$. This means that for all $\pp\in \Delta$, $v_{\pp}(x)<0$  and hence $v_{\pp}(x^{-1})>0$, giving $x^{-1}\in \bigcap_{\pp\in \Delta} \OO_{\pp} = J(R)$. Hence $x\not\in \widetilde{R}$. For the reverse inclusion, suppose $x\in \OO_{\pp}$ for some $\pp\in \Delta$. Then if $y\in J(R)$, we have that $v_{\pp}(x\cdot y)\geq 1$ and hence $x\cdot y \not=1$. Thus $x\in \widetilde{R}$. 
\end{proof}

Given a modulus $\mm$, let $S(\mm):=\{\pp: \pp|\mm\}$. 
\begin{thm}\label{univdefeqn}
For any global function field $K$ and finite set of primes $S\subset S_K$, with $\mm$ chosen as before,
\[
\OO_S = \bigcap_{\pp\in S(\mm)\setminus S} \widetilde{\OO_{\pp}} \cap \left(\bigcap_{\sigma\not=(1,1)}\bigcap_{p\in \Phi_{\sigma}} \widetilde{R_p^{\sigma}}\right) \cap \bigcap_{(p,q)\in \Psi_{K}} \widetilde{R_{p,q}^{(1,1)}},
\]
where $\Phi_{\sigma}$ and $\Psi_{K}$ are the diophantine sets in the previous section. 
\end{thm}

\begin{proof}
For $p\in \Phi_{\sigma}$ and $(p,q)\in \Psi_{K}$, all of the sets $\PP^{\sigma}(p)$ and $\Delta_{ap,q}\cap\Delta_{bp,q}$ are nonempty. Thus the right hand side is equal to 
\[
R_S:=\bigcap_{\pp\in S(\mm)\setminus S}\OO_{\pp} \cap \left(\bigcap_{\sigma\not=(1,1)}\bigcap_{\pp\in \Phi_{\sigma}}\bigcup_{\pp\in \PP^{\sigma}(p)} \OO_{\pp}\right) \cap \bigcap_{(p,q)\in \Psi_{K}}\bigcup_{\pp\in \Delta_{ap,q}\cap\Delta_{bp,q}} R_{p,q}^{(1,1)}.
\]

To show $R_S$ is contained in $\OO_S$, consider a prime $\pp_0\not \in S$. We need to show that any $x\in R_S$ is integral at $\pp_0$. If $\pp_0|\mm$, this is clear. If not, consider the image of $\pp_0$ under $\psi_{L/K}$.  First we assume $\psi_{L/K}(\pp_0)\not= (1,1)$. Then we claim that we can choose $p,p'\in \Phi_{\sigma}$ such that 
\[
\OO_{\pp_0} = \bigcup_{\pp\in \PP^{\sigma}(p)} \OO_{\pp} \cap \bigcup_{\pp\in \PP^{\sigma}(p')} \OO_{\pp}.
\]
If this claim is true, then $x\in R_S$ implies that $v_{\pp_0}(x)\geq 0$. Suppose $\sigma=(-1,-1)$. Using Lemma \ref{not1,1}, we find $p\in \Phi_{\sigma}$ such that $\{\pp_0\}=\PP^{\sigma}(p)$. Now let $\pp_1$ be some other prime of $K$ not dividing $\mm$ with $\psi_{L/K}(\pp_1)=(1,1)$, set $S'=\{\pp_1\}$ and $A:=\OO_{S'}$. Then using Lemma \ref{3.19}, we find a prime $\qq$ of $K$ such that $\qq\cap A$ is in the ideal class of $(\pp_0\cap A)^{-1}$ in $\Cl(A)$ and $\psi_{L/K}(\qq)=(1,1)$. There  exists an element $p'\in A$ such that $(\pp_0\cap A)(\qq\cap A)=p'A$. Then $p'\in \Phi_{(-1,-1)}$, as $\pp_0,\pp_1,$ and $\qq$ do not divide $\mm$, $\PP(p')=\{\pp_0,\pp_1,\qq\}$, and
\[
\psi_{L/K}(p')=\psi_{L/K}(\pp_0)\psi_{L/K}(\pp_1)\psi_{L/K}(\qq)=(-1,-1).
\]
Thus $\PP^{(-1,-1)}(p)\cap\PP^{(-1,-1)}(p)=\{\pp_0\}$ and the claim follows; the case of $\sigma=(-1,1)$ and $\sigma=(1,-1)$ is entirely similar. 

If $\pp_0$ satisfies $\psi_{L/K}(\pp_0)=(1,1)$ then we have seen in Lemma \ref{1,1} that there exist $(p,q)\in \Psi_K$ such that $\{\pp_0\}=\Delta_{ap,q}\cap \Delta_{bp,q}$ (and consequently $(p,q)\in \Psi_{K}$), implying
\[
\bigcup_{\pp\in \Delta_{ap,q}\cap \Delta_{bp,q}} \OO_{\pp} = \OO_{\pp_0}.
\]

Thus $x$ is integral at primes outside of $S$. 

To show the reverse inclusion, we claim that membership in $R_S$ imposes no integrality condition at a prime in $S$.  If $\pp_0\in S$, for any $\sigma\in \Gal(L/K)$ and any $p\in \Phi_{\sigma}$, we have $\pp_0\not\in \PP^{\sigma}(p)$. Additionally, $\pp_0\not\in \Delta_{ap,q}\cap\Delta_{bp,q}$ for $(p,q)\in \Psi_K$ as $\Delta_{ap,q}\cap\Delta_{bp,q}\subseteq I_{\mm}$ by Lemma \ref{1,1} part (2). 
\end{proof}

Now we are ready to prove our first main theorem.

\vspace{.5em}
\noindent{\bf Theorem \ref{univdefthm}.}
\emph{ For any global function field $K$ with $\ch(K)\not=2$, and any nonempty, finite set of primes $S$ of $K$, $\OO_S$ is defined by a first-order universal formula.}

\begin{proof}
Theorem \ref{univdefeqn} shows that for $t\in K$, 
\begin{align*}
t\in \OO_S \iff&  t \in \bigcap_{\pp|\mm} \widetilde{\OO_{\pp}}  \\
&\land\forall p \bigwedge_{\sigma\not=(1,1)} (p\not\in \Phi_{\sigma} \lor t\in \widetilde{R_p^{\sigma}}) \\
&\land\forall p,q \quad (p,q)\not\in \Psi_{K} \lor t\in \widetilde{R_{p,q}^{(1,1)}}.
\end{align*}
Given $\sigma\in \Gal(L/K)$, $p\in \Phi_{\sigma}$, and $(p,q)\in \Psi_{K}$, we have that $J(R^{\sigma}_p)$ and $J(R^{(1,1)}_{p,q})$ are diophantine by Lemmas \ref{not1,1} and \ref{1,1}. The sets $\Phi_{\sigma}$ and $\Psi_K$ are diophantine by the same lemmas, so their complements are defined by a universal formula.  Additionally, membership in $\widetilde{R_p^{\sigma}}$ is given by a universal formula, along with membership in $\widetilde{R_{p,q}^{(1,1)}}$. Hence membership in $\OO_S$ can be given by a universal formula.

\end{proof}

\section{Non-squares and non-norms of global fields are diophantine}
Now we let $K$ be a global field with $\ch(K)\not=2$.  Throughout this
section, fix $a,b,c,d$ as in Lemmas~\ref{3.19} and~\ref{choosec,d} if
$K$ is a global function field, and if $K$ is a number field, choose
$a,b\in K^{\times}$ as in Proposition 3.19 of
\cite{Park}. Additionally, we fix $c,d\in K^{\times}$ satisfying the
same assumptions as in Lemma~\ref{choosec,d} but now in the case that
$K$ is a number field. Corollary~3.20 in~\cite{Park} is incorrect as
stated, and has to be modified as in
Corollary~\ref{cor3.20} of our paper. With this modification, the other results
from~\cite{Park} are correct as stated. Set $L:=K(\sqrt{a},\sqrt{b})$ and fix an
admissible modulus $\mm$ for $L/K$ where $\mm$ contains all the primes
dividing $(2abcd)$, and all real places if $K$ is a number field. In
this section we prove Theorems~\ref{nonsq} and~\ref{nonnorm}.
\subsection{Non-squares}\label{section:nonsq}
We begin by proving Theorem \ref{nonsq}, i.e.\ that $K\setminus
K^{\times2}$ is diophantine over $K$. We use the fact that $x\not\in
K^{\times2}$ is equivalent to $x$ being a non-square in a completion
of $K$ at some prime $\pp$, where $\pp$ may be finite or
infinite. 

Suppose $\infty$ is a real archimedean prime of $K$ corresponding to an embedding $\omega:K\hookrightarrow \RR$. For $x\in K^{\times}$, we have that $x$ is not a square in $K_{\infty}$ if and only if $\omega(x)<0$. Therefore we require the following lemma:

\begin{lem}\label{positivedioph}
Suppose $K$ is a number field and that $\infty$ is a real archimedean prime of $K$ corresponding to an embedding $\omega: K \hookrightarrow K_{\pp}=\RR$. Then the set 
\[
K_{\omega}^{\times}:=\{x\in K^{\times}: \omega(x)>0\}
\]
is diophantine over $K$. Moreover, the set 
\[
\{x\in K^{\times}: x \text{ is not totally positive } \}
\]
is diophantine over $K$. 
\end{lem}

\begin{proof}
Let $\pp$ be a finite prime of $K$ and let $B$ be the quaternion
algebra over $K$ ramified exactly at $\pp$ and $\infty$. Let $\Nrd: B\to K$ denote the reduced norm map on $B$. Then by Facts I and II of \cite{HS36}, $\Nrd(B^{\times})=K_{\omega}^{\times}$. Let $b_1,b_2,b_3,b_4$ be a $K$-basis of $B$. Then 
\[
q(x_1,x_2,x_3,x_4):=\Nrd\left(\sum_{i=1}^4 x_ib_i\right)
\]
is a quadratic form in the variables $x_1,\ldots,x_4$ with coefficients in $K$, and 
\[
K_{\omega}^{\times} = \{x\in K^{\times}: \exists a_1,\ldots,a_4 \in K\text{ s.t. } q(a_1,a_2,a_3,a_4)=x\}.
\]
We conclude that $K_{\omega}^{\times}$ is diophantine over $K$. For an
alternative proof, see \cite{Den80}, Lemma 10. 

 Now we will prove the second statement. Let $\omega_1,\ldots,\omega_r$ denote a complete set of representatives of the inequivalent embeddings of $K$ into $\RR$. The set of elements which are not totally positive is diophantine over $K$, since it is the finite union
$\bigcup_{i} -K_{\omega_i}^{\times}.$

\end{proof}

The main lemma used in proving Theorem \ref{nonsq} is the following:
\begin{lem}\label{noreal}
Let $K$ be a global field with $\ch(K)\not=2$. Then
\[
x\not\in K^{\times2} \iff \begin{cases} x \text { is not totally
    positive, or } \\ v_{\pp}(x) \text{ is odd for some } \pp|\mm_0, \text{ or } \\ \exists p\in \Phi_{(-1,1)} \text{ such that } x\in a\cdot K^{\times2}\cdot(1+J(R_p^{(-1,1)})). \end{cases}
\]
\end{lem}
Here, $\mm_0$ denotes the finite primes dividing $\mm$, and $\Phi_{(-1,1)}$, $R_p^{(-1,1)}$, $a$, and $\mm$ are as in
Section
 \ref{integralatsigma}.
Before proving Lemma~\ref{noreal}, we will show how to use it to prove Theorem~\ref{nonsq}.

\begin{proof}[Proof of Theorem \ref{nonsq} from Lemma \ref{noreal}]
We begin by showing the right side of the equivalence in Lemma
\ref{noreal} defines a diophantine subset of $K$. We claim that, for a prime $\pp$ of $K$, the set
\[
\{x\in K^{\times}: v_{\pp}(x) \text{ is odd}\} 
\]
is diophantine over $K$. To see this, fix $p\in K^{\times}$ with
$v_{\pp}(p)=1$. Then $v_{\pp}(x)$ is odd if and only if 
$x\in p\cdot K^{\times2}\cdot \OO_{\pp}^{\times}$. This is diophantine over $K$ since
\[
\OO^{\times}_{\pp}=\OO_{\pp}\cap\{x\in K^{\times}: x^{-1}\in \OO_{\pp}\}.
\]
Since only finitely many primes divide $\mm_0$, the condition that 
$v_{\pp}(x)$ is odd for some $\pp|\mm_0$ is diophantine as well. As
noted in Lemma \ref{positivedioph}, the condition that $x$ is not
 totally positive defines a diophantine set. If $K$ is a global
 function field, Lemma \ref{not1,1} parts (1) and (2) complete the
 argument that the right side of Lemma \ref{noreal} defines a
 diophantine subset of $K$. The number field case follows from parts
 (a) and (b) of Lemma 3.22 of \cite{Park}. Hence Lemma~\ref{noreal} 
shows that $K^{\times}\setminus K^{\times2}$ is diophantine.
\end{proof}

We are now left with proving Lemma \ref{noreal}, for which we need the following lemma.

\begin{lem}\label{nslem1}
Let $\Phi_{(-1,1)}$, $R_p^{(-1,1)}$, $a$, $b$, and $\mm$ be defined as in Section \ref{integralatsigma}. Let $p\in \Phi_{(-1,1)}$. Then
 $x\in a\cdot K^{\times2}\cdot(1+J(R_p^{(-1,1)}))$ if and only if there exists an element $t\in K^{\times}$ such that $\forall \qq\in \PP^{(-1,1)}(p)$, $v_{\qq}(xt^2)=0$ and the image of $xt^2$ in the residue field of $\qq$ is not a square. 
\end{lem}

\begin{proof}
Since $p\in \Phi_{(-1,1)}$, we have that $\PP^{(-1,1)}(p)\not=\emptyset$. In the function field setting, this follows from Lemma \ref{not1,1}, and in the number field setting, it follows from Lemma 3.22 of \cite{Park}. We also have that 
\begin{align*}
R_p^{(-1,1)}&=\bigcap_{\qq\in \PP^{(-1,1)}(p)} \OO_{\qq}, \\ 
J(R_p^{(-1,1)})&=\bigcap_{\qq\in \PP^{(-1,1)}(p)} \qq O_{\qq}.
\end{align*}
Suppose there exists $t\in K^{\times}$ as in the lemma. Let $\qq\in \PP^{(-1,1)}(p)$. Then $xt^2\equiv a$ in  $(\OO_{\qq}/\qq)^{\times}/(\OO_{\qq}/\qq)^{\times2}$, since by our choice of $a$ and $b$,
\[
\left(\left(\frac{a}{\qq}\right)_2,\left(\frac{b}{\qq}\right)_2\right) = (\qq,K(\sqrt{a},\sqrt{b})/K)=(-1,1).
\]
Here $\left(\frac{\cdot}{\qq}\right)_2$ is the degree $2$ power residue symbol for $K$. Thus there is some $s_{\qq}\in (\OO_{\qq})^{\times}$ such that $xt^2=as_{\qq}^2 \mod \qq$; using the Chinese remainder theorem we find an $s$ so that $xt^2=as^2\mod \qq$ for each $\qq$. Since $q:=xt^2-as^2\in \qq$ for each $\qq\in \PP^{(-1,1)}(p)$, it is in their intersection. Thus
\[
x=a\cdot(s/t)^2(1+q/(as^2)).
\]
We claim $x$ is in $a\cdot K^{\times2}\cdot (1+\bigcap_{\qq\in \PP^{(-1,1)}(p)} \qq\OO_{\qq})$. Let $\qq\in \PP^{(-1,1)}(p)$. Then $v_{\qq}(a)=0$ because $(a)$ is only divisible by primes dividing $\mm$ and $\qq\in I_{\mm}$. Additionally, $v_{\qq}(s)=0$ and $v_{\qq}(q)\geq 1$, because $s\in \OO_{\qq}^{\times}$ and $q \in \qq$. The claim follows as 
\[
v_{\qq}(q/as^2)\geq 1-v_{\qq}(a)-2v_{\qq}(s)\geq1.
\]

For the other implication, write $x=at^2(1+j)$ with $t\not=0$ and $j\in \bigcap_{\qq\in \PP^{(-1,1)}(p)} \qq \OO_{\qq}$. Then 
\[
v_{\qq}(x/t^2)=v_{\qq}(a(1+j))=v_{\qq}(a)=0
\]
for each $\qq\in \PP^{\sigma}(p)$, and $x/t^2\equiv a$ modulo $\qq$ for $\qq\in \PP^{(-1,1)}(p)$. Finally, by construction, $a$ is a non-square in the completion at any prime $\qq\in \PP^{(-1,1)}(p)$.
\end{proof}

We now prove Lemma \ref{noreal}:
\begin{proof}[Proof of Lemma \ref{noreal}]
  For the forward implication, suppose $x$ is not a square. If $x$ is not totally positive, or $v_{\pp}(x)$ is odd for some $\pp|\mm$, we are done, so assume that $v_{\pp}(x)$ is even for each
  $\pp|\mm$ and that $x$ is totally positive. Using weak approximation, we may assume that in fact $v_\pp(x)=0$ for each
  $\pp|\mm$. This does not change the truth value of either
  side of the implication in Lemma \ref{noreal}, as both statements are invariant under
  multiplying $x$ by a square.

What we will now show is that there is a $p\in \Phi_{(-1,1)}$ such
that $\PP^{(-1,1)}(p)=\{\qq\}$, $v_{\qq}(x)=0$, and such that $x$ modulo $\qq$
is not a square. Together with Lemma \ref{nslem1}, this will imply that $x\in a\cdot
K^{\times2}\cdot(1+J(R_p^{(-1,1)}))$.

As $K(\sqrt{x})$ is a degree $2$ extension of $K$ unramified
at all $\pp|\mm$, it is linearly disjoint from
$L=K(\sqrt{a},\sqrt{b})$ over $K$. Let $\tau$ be the nontrivial
automorphism of $K(\sqrt{x})/K$ and consider $(\tau,(-1,1))\in
\Gal(K(\sqrt{x})/K)\times\Gal(L/K)$. By the
Chebotarev Density Theorem, there is a prime $\qq$ of $K$ so that its
associated Frobenius automorphism is $(\tau,(-1,1))$. The
restriction of this automorphism to $K(\sqrt{x})$ is
\[
\tau=(\qq,K(\sqrt{x})/K),
\]
implying $\qq$ does not split completely in $K(\sqrt{x})$. Hence $x$ is not a square
in the completion of $K$ at $\qq$. The restriction to $L$ is
\[
(-1,1)=(\qq,K(\sqrt{a},\sqrt{b})/K),
\]
so we have that $\qq\in \PP^{(-1,1)}$. If $K$ is a global function field, Lemma \ref{not1,1} implies there exists an element $p\in
K^{\times}$ so that $\{\qq\}= \PP^{(-1,1)}(p)$. In the number field setting, this follows from Lemma 3.22 in
\cite{Park}. This is the desired element $p$ and prime $\qq$, and finishes the first half of the proof.

For the reverse implication, suppose that there is an element $p\in \Phi_{(-1,1)}$ such that $x\in a\cdot K^{\times2}\cdot (1+J(R^{(-1,1)}_p))$. By Lemma \ref{nslem1} there exists $t\not=0$ such that $xt^2$ is a non-square in the completion of $K$ at each prime $\qq\in \PP^{(-1,1)}(p)$. This implies $xt^2$ itself is not a square in $K$, i.e.\ $x\not\in K^{\times2}$. 

\end{proof}

\subsection{Non-norms}\label{section:non-norms}
Let $K$ be a global field with $\ch(K)\not=2$. To prove Theorem \ref{nonnorm}, which states that the set
\[
\{(x,y)\in K^{\times}\times K^{\times}: x \text{ is not a norm of } K(\sqrt{y})\}
\]
is diophantine over $K$, we use the Hasse
norm theorem: for a cyclic extension $L/K$, an element of $K$ is a
(relative) norm of an element of $L$ if and only if it is a local norm in
every completion of $K$. We will use the fact that
\[
(x,y)_{\pp}=-1 \Leftrightarrow x \text{ is not a local norm of } K(\sqrt{y}) \Leftrightarrow y \text{ is not a local norm of } K(\sqrt{x}).
\]

We begin by establishing that, given a finite or infinite prime $\pp$ of $K$, the collection of pairs $(x,y)\in K^{\times}\times K^{\times}$ such that $x$ is not a local norm in the completion of $K(\sqrt{y})$ at $\pp$ is diophantine over ~$K$. 

First we need the following lemma, which will be used both in the proof of Theorem \ref{nonlocnorm} and later in proving Theorem \ref{nonnorm}.

\begin{lem}\label{hsodd}
Let $\pp$ be a finite prime of $K$ with $|\FF_{\pp}|$ odd. Fix $p,s\in K^{\times}$ so that $v_{\pp}(p)=1$ and $v_{\pp}(s)=0$ so that $\red_{\pp}(s)$ is not a square in the residue field of $\pp$. Then for $x,y\in K^{\times}$, $(x,y)_{\pp}=-1$ if and only if
\begin{align*}
\big((x\in p\cdot K^{\times2}\cdot\OO_{\pp}^{\times}) &\land (y \text{ or } -xy \in s\cdot K^{\times2}  \cdot(1+\pp\OO_{\pp}))\big) \\ 
\lor \big((y\in p\cdot K^{\times2} \cdot \OO_{\pp}^{\times}) &\land (x \text{ or } -xy \in s\cdot K^{\times2} \cdot(1+\pp\OO_{\pp}))\big).
\end{align*}
\end{lem}

\begin{proof}

Assume that $(x,y)_{\pp}=-1$. From the formula for the Hilbert symbol in Section \ref{QAfacts}, at least one of $x$ or $y$ must have odd valuation at $\pp$, so without loss of generality, assume $v_{\pp}(x)$ is odd. As $(x,y)_{\pp}=-1$, the formula for the Hilbert symbol implies $((-1)^{v_{\pp}(x)v_{\pp}(y)})^{(|\FF_{\pp}|-1)/2}=-1$ and $\text{red}_{\pp}\left(\frac{x^{v_{\pp}(y)}}{y^{v_{\pp}(x)}}\right)^{\frac{|\FF_{\pp}|-1}{2}}
=1$ or vice versa. Then $v_{\pp}(x/p)$ is even and $x/p\in K^{\times2}\cdot \OO_{\pp}^{\times}$.

\noindent{\bf Case 1.} If $v_{\pp}(y)$ is even, then using weak approximation, choose $t\in K^{\times}$ so that $yt^2$ is a $\pp$-adic unit. In this case, 
$\text{red}_{\pp}\left(\frac{x^{v_{\pp}(y)}}{y^{v_{\pp}(x)}}\right)^{\frac{|\FF_{\pp}|-1}{2}}=-1$ since $(-1)^{v_{\pp}(x)v_{\pp}(y)}=1$.
We claim that $yt^2$ is a non-square modulo $\pp$. This follows from 

\[
\left(\frac{x^{v_{\pp}(y)}y^{-v_{\pp}(x)}yt^2}{\pp}\right)_2=\left(\frac{(x^{(v_{\pp}(y))/2}y^{(1-v_{\pp}(x))/2}t)^2}{\pp}\right)_2=1.
\]
Since $yt^2$ is not a square mod $\pp$, an argument similar to the one in the proof of Lemma \ref{nslem1} shows that $y\in ~s\cdot K^{\times2}\cdot(1+\pp\OO_{\pp})$. 

\noindent{\bf Case 2.} Suppose $v_{\pp}(y)$ is odd. We claim that, possibly after multiplying by a square of $K$, $-xy$ is not a square modulo $\pp$. This would imply that
\[
-xy\in s\cdot(K^{\times})^2(1+\pp\OO_{\pp}),
\] 
which is what we want to show. To prove the claim, use weak approximation to find $t\in K^{\times}$ such that $v_{\pp}(xyt^2)=0$. The following calculation shows $x^{v_{\pp}(y)}/y^{v_{\pp}(x)}$ and  $xyt^2$ have the same $2$-power residue for the prime $\pp$:
\[
\left(\frac{x^{v_{\pp}(y)}y^{-v_{\pp}(x)}xyt^2}{\pp}\right)_2=\left(\frac{(x^{(v_{\pp}(y)+1)/2}y^{(1-v_{\pp}(x))/2}t)^2}{\pp}\right)_2=1.
\]
 If $|\FF_{\pp}|$ is 1 modulo $4$, then $-1$ is a square in $\FF_\pp$. From the formula for the Hilbert symbol, we must have that $x^{v_{\pp}(y)}/y^{v_{\pp}(x)}$ is not a square in the residue field of $\pp$, and hence neither is $-xyt^2$. If $|\FF_{\pp}|$ is $3 \mod 4$, then considering the formula for $(x,y)_{\pp}$ again, we have $(-1)^{(|\FF_{\pp}|-1)/2}=-1$ and thus $(x,y)_{\pp}=-1$ implies $xyt^2$ is a square modulo $\pp$. But since $-1$ is not a square modulo $\pp$, this implies $xyt^2$ is not a square modulo $\pp$, and hence $-xy\in a\cdot K^{\times2}\cdot(1+\pp\OO_{\pp})$.

Conversely, if $x\in p\cdot K^{\times2}\cdot\OO_{\pp}^{\times}$, then $v_{\pp}(x)$ is odd. If $y\in s\cdot K^{\times2}  \cdot(1+\pp\OO_{\pp})$, then for some $t\in K^{\times}$, $yt^2$ is a $\pp$-adic unit and is a non-square modulo $\pp$. If $-xy\in s\cdot K^{\times2}  \cdot(1+\pp\OO_{\pp})$, then, possibly after multiplying by a square of $K^{\times}$, $-xy$ is a $\pp$-adic unit and is a non-square modulo $\pp$. Hence $(x,-xy)_{\pp}=-1$. Since $(x,-x)_{\pp}=1$,
\[
(x,y)_{\pp}=(x,-xy)_{\pp}=-1.
\]
\end{proof}

\begin{thm}\label{nonlocnorm}
Let $\pp$ be a finite or real infinite prime of $K$. The set $\{(x,y)\in K^{\times}\times K^{\times}: (x,y)_{\pp}=-1\}$ is diophantine over $K$.
\end{thm}

\begin{proof}
First assume $\pp$ corresponds to a real archimedean absolute value on
$K$. Let $\omega:K\hookrightarrow \RR$ be the corresponding embedding
of $K$ into $K_{\pp}=\RR$. Then $(x,y)_{\pp}=-1$ if and only if the
equation $xs^2+yt^2=1$ has no solutions in $\RR\times \RR$, which
holds if and only if $\omega(x)<0$ and $\omega(y)<0$. These conditions
are diophantine by Lemma \ref{positivedioph}. 

Now suppose $\pp$ is a finite prime of $K$. If $|\FF_{\pp}|$ is odd, the lemma follows from Lemma \ref{hsodd} since all the sets appearing are diophantine over $K$. Since our statements are only for global fields $K$ with $\ch(K)\not=2$, the only remaining case is that $K$ is a number field and $\pp|2$. 

First we show that there are only finitely many elements in $K_{\pp}^{\times}/K_{\pp}^{\times2}$. Let $\pi$ be a uniformizer for $\hat{\pp}$ and let $e$ be the absolute ramification index, meaning $(\pi)^e=(2)$. Then if $\alpha\in R_{\pp}^{\times}$ is in $1+\hat{\pp}^{2e+1}$, it is a square in $R_{\pp}^{\times}$ by Hensel's Lemma. We conclude that $R_{\pp}^{\times2}$ is open in the profinite group $R_{\pp}^{\times}$ since it contains $1+\hat{\pp}^{2e+1}$, a neighborhood of $1$. As open subgroups of compact groups have finite index, we conclude that $R_{\pp}^{\times2}$ has finite index in $R_{\pp}^{\times}$. To see that $[K_{\pp}^{\times}:K_{\pp}^{\times2}]$ is finite, we now just need to observe that squares of $K_{\pp}^{\times}$ are of the form $s\cdot\pi^{2k}$ where $s\in R_{\pp}^{\times2}$.

Let $s_1,\ldots,s_n\in K$ be a set of representatives for $K_{\pp}^{\times}/K_{\pp}^{\times2}$ and define $S_j:= s_j\cdot K^{\times2}\cdot(1+\pp^{2e+1}\OO_{\pp})$. For $x\in S_i$, $y\in S_j$ we have $(x,y)_{\pp}=(s_i,s_j)_{\pp}$.

Now we define
\[
S_{\pp}:=\bigcup_{i,j: (s_i,s_j)_{\pp}=-1} S_i\times S_j.
\]
Then $(x,y)_{\pp}=-1$ if and only if $(x,y)\in S_{\pp}$. Each set $S_i$ is diophantine over $K$, and the finite Cartesian product of diophantine sets is again diophantine. Thus $S_{\pp}$ is diophantine over $K$.
  
\end{proof}

Now we prove Theorem \ref{nonnorm} for a global field $K$ with $\ch(K)\not=2$.

\begin{proof}[Proof of Theorem \ref{nonnorm}]
  By the Hasse norm principle, $x$
  is not a norm in $K(\sqrt{y})$ if and only if $(x,y)_{\pp}=-1$ for
  some finite or real infinite prime $\pp$ of $K$. For $\sigma\not=(1,1)\in
  \Gal(K(\sqrt{a},\sqrt{b})/K)$, let $s_{\sigma}=a$ if
  $\sigma=(-1,\pm1)$ and $s_{\sigma}=b$ if $\sigma=(1,-1)$. We claim
  that $(x,y)_{\pp}=-1$ if and only if one of the following conditions
  holds:

\begin{itemize}
\item $\exists\, \pp|\mm \text{ such that } (x,y)_{\pp}=-1$,
\item $\bigvee_{\sigma\not=(1,1)} \exists p \in \Phi_{\sigma} \text{ such that}$
\begin{align*}
\big((x\in p \cdot K^{\times2} \cdot (R_{p}^{\sigma})^{\times}) &\land (y \text{ or } -xy \in s_{\sigma} \cdot K^{\times2} \cdot (1+J(R^{\sigma}_p)))\big) \\ 
\lor \big((y\in p \cdot K^{\times2} \cdot (R_{p}^{\sigma})^{\times}) &\land (x \text{ or } -xy \in s_{\sigma} \cdot K^{\times2} \cdot (1+J(R^{\sigma}_p)))\big),
\end{align*}
\item $\exists(p,q)\in \Psi_K \text{ such that } q\in (R^{(1,1)}_{p,q})^{\times} \text{ and }$ 
\begin{align*}
\big((x\in p \cdot K^{\times2} \cdot (R_{p,q}^{(1,1)})^{\times}) &\land (y \text{ or } -xy \in q \cdot K^{\times2} \cdot (1+J(R^{(1,1)}_{p,q})))\big) \\ 
\lor\big((y\in p \cdot K^{\times2} \cdot (R_{p,q}^{(1,1)})^{\times}) &\land (x \text{ or } -xy \in q \cdot K^{\times2} \cdot (1+J(R^{(1,1)}_{p,q})))\big).
\end{align*}

\end{itemize}

This will imply the theorem, because we have already shown that the above conditions define diophantine sets. We will first prove the forward implication. If $x$ is not a norm in
$K(\sqrt{y})$, there is a prime $\pp$ of $K$ such that
$(x,y)_{\pp}=-1$. If $\pp|\mm$ we are
done; recall that $\mm$ contains all real infinite primes if $K$ is a
number field.

Now assume $\pp \in I_{\mm}$ and that $\psi_{L/K}(\pp)=\sigma\not=(1,1)$. We
claim that the second condition holds. We can find $p\in \Phi_{\sigma}$
such that $\PP^{\sigma}(p)=\{\pp\}$ as before. Corollary \ref{cor3.20} and the definition of $R_p^{\sigma}$ imply
$R_p^{\sigma}=\OO_{\pp}$ and $J(R_p^{\sigma})=\pp\OO_{\pp}$. By setting $p:=p$ and $s:=s_{\sigma}$, Lemma
\ref{hsodd} implies that the second condition holds because $v_{\pp}(p)$ is odd
and $s_{\sigma}$, by construction, is a $\pp$-adic unit which is not a square modulo
$\pp$. For example, if $\sigma=(1,-1)$, then $s_{\sigma}=b$. The fractional ideal $(b)$ is coprime to $\mm$, and $\psi_{L/K}(\pp)=(1,-1)$ implies that $b$ is not a square mod $\pp$. 

Now assume $\pp \in I_{\mm}$ with $\psi_{L/K}(\pp)=(1,1)$. We claim that the third condition holds
in this case. We will first show that we can find $(p,q)\in
\Psi_K$ with the stated properties if $K$ is a global function field. By Lemma \ref{1,1}, there is a
$(p,q)\in \Psi_K$ with $\Delta_{ap,q}\cap \Delta_{bp,q}=\{\pp\}$. In fact,
 $q$ can be chosen so that $v_{\pp}(q)=0$, $q$ is not a square modulo $\pp$, and such that $\PP(q)=\{\qq,\pp_0\}$. Here, $\qq$ and $\pp_0$ are primes with 
$\psi_{L/K}(\qq)=(-1,-1)$ and $\psi_{L/K}(\pp_0)=(1,1)$. Then $q\in
\OO_{\pp}^{\times}=(R^{(1,1)}_{p,q})^{\times}$ by Definition \ref{semilocalringsdef}. By the formula for the Hilbert symbol,
$v_{\pp}(ap)$ is odd, and since $\pp$ cannot divide $(a)$, we conclude $v_{\pp}(p)$ is odd. As $R_{p,q}^{(1,1)}=\OO_{\pp}$ and
$J(R_{p,q}^{(1,1)})=\pp\OO_{\pp}$, we can apply Lemma \ref{hsodd} with
$s:=q$. 

If $K$ is a number field, then by Lemma 3.25 in \cite{Park},
we can find $(p,q)\in \Psi_K$ such that $(q)$ is a prime ideal with
$\psi_{L/K}((q))=(-1,-1)$, $q$ is not a square in $K_{\pp}$, and such that
$\Delta_{ap,q}\cap\Delta_{bp,q}=\{\pp\}$. A similar argument as above
shows that the second condition holds.

Now we will prove that if one of the three conditions above holds, then for
some prime $\pp$, $(x,y)_{\pp}=-1$. Suppose the second condition
holds for some $\sigma\not=(1,1)$. If $K$ is a global function field, $\PP^{\sigma}(p)\not=\emptyset$ by Lemma \ref{not1,1} (2), so $\PP^{\sigma}(p)$ contains some
prime $\pp$. Assume, without loss of generality, that
\[
\big((x\in p \cdot K^{\times2} \cdot (R_{p}^{\sigma})^{\times}) \land (y \text{ or } -xy \in s_{\sigma} \cdot K^{\times2} \cdot (1+J(R^{\sigma}_p)))\big).
\]
Since $v_{\pp}(p)$ is odd, $v_{\pp}(x)$ is odd,
too. Also, either $y$ or $-xy$ is, possibly after multiplying by a
square of $K^{\times}$, a non-square in $K_{\pp}$ since
$\psi_{L/K}(\pp)=\sigma$ implies $s_{\sigma}$ is a non-square in
$K_{\pp}$. By Lemma \ref{hsodd}, this implies that
$(x,y)_{\pp}=-1$. If $K$ is a number field, an application of Lemma 3.22 (b) in \cite{Park} and a similar argument show that $(x,y)_{\pp}=-1$ for a finite prime $\pp$ of $K$. 

We now prove that the third condition implies that $(x,y)_{\pp}=-1$ for some
$\pp$ with $\psi_{L/K}(\pp)=(1,1)$. The argument is similar to the one in the second condition. If $K$ is a global function field, $(p,q)\in \Psi_K$ implies that 
$\Delta_{ap,q}\cap\Delta_{bp,q}\not=\emptyset$ and contains some prime $\pp$ by Lemma \ref{1,1} part (2). Then because $q\in (R_{p,q})^{\times}$ and
$(ap,q)_{\pp}=(bp,q)_{\pp}=-1$, $q$ is not a square mod $\pp$ and $v_{\pp}(p)$ must be odd. Again by Lemma
\ref{hsodd}, this implies that $(x,y)_{\pp}=-1$. If $K$ is a number field, the same argument, along with Lemma 3.25 (b) from \cite{Park}, proves the claim.

\end{proof}

\section*{Acknowledgments}
The first author was partially supported by National
  Science Foundation grant DMS-1056703.
The second author was partially
  supported by  National
  Science Foundation grants DMS-1056703 and CNS-1617802.


\end{document}